\documentclass{amsart}
\usepackage{color}
\usepackage[utf8]{inputenc}
\usepackage{amscd}
\usepackage{pstricks}
\usepackage{amsmath}
\usepackage[english]{babel}
\usepackage{amsfonts}
\usepackage{graphicx}
\usepackage{vmargin}
\usepackage{theoremref}
\usepackage[hidelinks]{hyperref}
\usepackage{amssymb}
\usepackage{tikz}
\usepackage{relsize}
\usepackage[toc,page]{appendix}
\usepackage{commath}
\usepackage{mathrsfs}
\usetikzlibrary{matrix}
\usepackage{cancel}
\usepackage{lineno}
\usepackage{comment}
\usepackage{refcheck}

\newcommand{\dist}{\mathrm{dist}}
\newcommand{\NmuD}{{\mathrm{N}\mu\mathrm{D}}}

\newcommand{\wxi}{{\widetilde{\xi}}}
\newcommand{\wgamma}{{\widetilde{\gamma}}}
\newcommand{\muD}{{\mu\mathrm{D}}}
\newcommand{\NPD}{\mathrm{NPD}}
\newcommand{\PD}{\mathrm{PD}}

\newcommand{\Si}{\Sigma}
\newcommand{\ED}{\mathrm{ED}}
\newcommand{\Id}{\mathrm{Id}}

\newcommand{\wmu}{\mathfrak{u}}
\newcommand{\wnu}{\mathfrak{v}}

\renewcommand{\P}{\mathrm{P}}
\renewcommand{\k}{\mathbf{k}}
\newcommand{\m}{\mathbf{m}}

\renewcommand{\a}{\hat{a}}
\renewcommand{\b}{\hat{b}}
\newcommand{\sgn}{\mathrm{sgn}}

\newcommand{\Dom}{\mathrm{Dom}}

\newtheorem{theorem}{Theorem}[section]
\newtheorem{lemma}[theorem]{Lemma}

\newtheorem{definition}[theorem]{Definition}

\newtheorem{corollary}[theorem]{Corollary}

\newtheorem{example}[theorem]{Example}
\newtheorem{notation}[theorem]{Notation}
\newtheorem{remark}[theorem]{Remark}

\usepackage{chngcntr}
\counterwithin{figure}{section}

\numberwithin{equation}{section}

\title[]{A generalization of Siegmund's normal forms theorem to systems with $\mu$-dichotomies}
\author[]{\'Alvaro Casta\~neda and N\'estor Jara}

\address{Universidad de Chile, Departamento de Matem\'aticas. Casilla 653, Santiago, Chile}

\email{castaneda@uchile.cl, nestor.jara@ug.uchile.cl}
\thanks{This research has been partially supported by FONDECYT Regular 1200653.}
\thanks{This research has been partially supported by ANID, Beca de Doctorado Nacional 21220105.}
\date{\today}

\begin{document}

\maketitle

\begin{abstract}
We establish a theorem concerning the normal forms by examining the newly presented concept of $\mu$-dichotomy. This work establishes the nonresonance condition based on the associated spectrum of this general nonautonomous hyperbolicity.
\end{abstract}

\section{Introduction}

Since the contributions of H. Poincar\'e \cite{Poincare} , there has been a significant focus on the dynamical systems of the normal forms.  This technique allows us to face the problem of linearizing a nonlinear system in the neighborhood of an equilibrium point; to address this question in a context of analytic linearization for analytic functions, the author introduces a condition referred to as the \textit{nonresonant condition} which essentially states that the eigenvalues of the linearization around the equilibrium point satisfy 
$\lambda_1\neq \sum_{j=1}^dm_j\lambda_j$, for all $m_1,\dots,m_d\in \mathbb{N}$ with $2\leq \sum_{j=1}^dm_j.$ Recall that in the case of $C^k$ vector fields, the $C^r$ linearization, with $1 \leq r \leq k \leq +\infty,$ was given by S. Sternberg in \cite{Sternberg1, Sternberg2} who regarded the criterion $2\leq \sum_{j=1}^dm_j \leq k$ instead of the one utilized by Poincar\'e. 

Until the end of the previous century, the focus of normal forms theory primarily revolved around autonomous differentiable systems. The core results of this theory may be found in a comprehensive manner on V.I. Arnold \cite[Chapter 5]{Arnold}, L. Stolovitch \cite{Stolovitch} and S. Wiggins \cite[Chapter 19]{Wiggins}.

\subsection{Nonautonomous Formal Norms} In \cite{Siegmund3}, S. Siegmund made a notable breakthrough by expanding on Poincaré's result using the spectrum associated with the exponential dichotomy, which may be interpreted as a kind of hyperbolicity in a nonautonomous context. Later, in \cite{Cuong}, the same author joint with L.V. Cuong and T.S. Doan extend the Sternberg Theorem to the context of nonautonomous differential systems. Notice that both previous work, the nonresonce condition is fashioned in terms of the Sacker-Sell spectrum (see \cite{Sacker, Siegmund}). 

In \cite{Zhang}, X. Zhang examines the nonuniform exponential dichotomy and establishes a normal form theorem inside this nonuniform framework, assuming nonresonance; this result is presented in terms of the spectrum of this dichotomy. Furthermore, it is important to note that in the work done by J. Chu \textit{et al.} \cite{Chu}, the spectrum of this nonautonomous hyperbolicity is also formulated.

\subsection{Structure and novelty of the article} In a recent study, C. Silva \cite{Silva} analizes the concept of $\mu$-dichotomy, which extends beyond prior dichotomies. Additionally, the author constructs the spectrum linked to this novel dichotomy.  The main objective of this work is to develop the theory of normal forms for a differentiable system that exhibits $\mu$-dichotomy; this will be done by establishing nonresonance condition based on the spectrum associated with this nonautonomous hyperbolicty.

\smallskip
In Section 2, we establish the fundamental components necessary for this study. We establish the notations and significant definitions for both the linear and nonlinear equations under investigation. In addition, we provide clear definition for the concept of equivalence and develop the premise that we will use to demonstrate this notion. 

Our main results are stated and proven in Section 3. Firstly, we provide the result of elimination nonresonant Taylor terms; this is done under the assumption that the linear component of the system exhibits a concept of uniformly bounded growth, while the nonlinearities are considered to be nonuniformly admissible (detailed definitions and examples may be found in Section 2). Secondly, we demonstrate the enhancement of this result when the nonlinearities are uniformly admissible, enabling us to derive the normal forms theorem. 

In Section 4, we provide a way to address this problem when the linear component exhibits only nonuniform bounded growth. We also discuss the challenges that arise in this scenario when attempting to produce a result of normal forms.

\section{Preliminaries and contextualization}

Let us proceed by establishing suitable notations for this work:
\begin{notation}
    {\rm 
    For a map $f$ defined on some region contained on $\mathbb{R}\times \mathbb{R}^d$ and taking values on $\mathbb{R}^m$ (some $m\in \mathbb{N}$), we write 
    \begin{itemize}
        \item $\Dom_f$ for its domain. If the function is evident, we just write $\Dom$.
        \end{itemize}

When they are well defined, we adopt the following notation for derivatives of $f$.
        \begin{itemize}
        \item $D_1f(t,x)$ is the derivative of $f$ respect to the {\it temporal} variable of its domain ({\it i.e.} $t\in \mathbb{R}$), evaluated on a point $(t,x)\in \Dom$.
        \item $D_2f(t,x)$ denotes the Jacobian of $f$, {\it i.e.} the differential respect to its  {\it spacial} variable (that is, $x\in \mathbb{R}^d$), evaluated on a point $(t,x)\in \Dom$.
        \item If $\mathscr{W}_1,\dots,\mathscr{W}_n$ are manifolds such that $\mathscr{W}_1\oplus \cdots\oplus \mathscr{W}_n=\mathbb{R}^d$, we denote the Jacobian of $f$ respect to the manifold $\mathscr{W}_i$ by $D_{2_{\mathscr{W}_i}}f(t,x)$. If the manifold decomposition is clear, we just write $D_{2_i}f(t,x)$.
        \item If the manifold $\mathscr{W}_i$ is decomposed on $d_i$ directions, denote the partial derivative of $f$ respect to the $j$-th component of the $i$-th manifold by $D_{2_{i,j}}f(t,x)$. 
    \end{itemize}
    We extend this notation for higher order derivatives. Finally, for functions $\psi:\mathbb{R}\to \mathbb{R}^m$ and $\varphi:\mathbb{R}^d\to\mathbb{R}^m$, we write
    \begin{itemize}
        \item $\left.D_t\left[\psi(t)\right]\right|_{t=\widetilde{t}}$ for the derivative of the function $\psi$ respect to $t$, evaluated on the point $\widetilde{t}$.
            \item $\left.D_x\left[\varphi(x)\right]\right|_{x=\widetilde{x}}$ for the differential of the function  $\varphi$ respect to $x$, evaluated on the $\widetilde{x}$.
    \end{itemize}
    
   We privilege these last notations when the functions are written as a composition or operations between other functions. In both cases, if we mean the {\it function} derivative and not a specific evaluation, we omit the indication of the evaluation in the sub index.
    }
\end{notation}

\subsection{The linear part}

We study a linear equation of the form
\begin{equation}\label{513}
        \dot{x}=A(t)x(t),
    \end{equation}
where $t\mapsto A(t)$ is locally integrable. Denote its evolution operator by $\Phi:\mathbb{R}\times \mathbb{R}\to \mathbb{R}^d$.

\begin{definition}
\cite[p. 621]{Silva}    We say a function $\mu:\mathbb{R}\to \mathbb{R}^+$ is a \textbf{growth rate} if it is strictly increasing, $\mu(0)=1$, $\lim_{t\to +\infty}\mu(t)=+\infty$ and $\lim_{t\to-\infty}\mu(t)=0$. If moreover  $\mu$ is differentiable, we say it is a \textbf{differentiable growth rate}.
\end{definition}

\begin{definition}\cite[p. 621]{Silva}  
    Denote the sign of $a\in \mathbb{R}$ by $\sgn(a)$. Let $\mu:\mathbb{R}\to \mathbb{R}$ be a growth rate. The system (\ref{513}) admits \textbf{nonuniform $\mu$-dichotomy} ($\NmuD$) if  there exist an invariant projector $t\mapsto \P(t)$ for (\ref{513}) and constants $K\geq 1$, $\alpha<0$, $\beta>0$ and $\theta,\nu\geq 0$ such that $\alpha+\theta<0$, $\beta-\nu>0$ and 
    \begin{align*}
            \norm{\Phi(t,s)\P(s)}&\leq K\left(\frac{\mu(t)}{\mu(s)}\right)^\alpha\mu(s)^{\sgn(s)\theta}\,\,\,\quad\text{ for }t\geq s,\nonumber\\
            \norm{\Phi(t,s)[\Id-\P(s)]}&\leq K\left(\frac{\mu(t)}{\mu(s)}\right)^\beta\mu(s)^{\sgn(s)\nu}\qquad\text{ for }t\leq s.
        \end{align*}
        If moreover $\theta=\nu=0$, then we say (\ref{513}) admits \textbf{uniform $\mu$-dichotomy ($\muD$)}.
\end{definition}

The following cases, among many others, are covered by this definition
\begin{itemize}
        \item [a)] Taking $\mu(t)=e^t$, $\beta=-\alpha$ and $\theta=\nu=0$, we recover the exponential dichotomy ($\ED$) defined by O. Perron \cite{Perron} and widely used on literature with spectral purposes \cite{Siegmund,Siegmund2,Siegmund3}.
        \item [b)] Taking just $\mu(t)=e^t$, we obtain the \textbf{nonuniform exponential dichotomy}, also widely studied \cite{Chu,Dragicevic2,Zhang}.
        \item [d)] For a strictly increasing surjective function $\nu:\mathbb{R}_0^+\to[1,+\infty)$ we can define a growth rate by
         \begin{equation*}
    \mu(t)= \left\{ \begin{array}{lcc}
             \nu(t) &  \text{ if } &   t\geq 0, \\
             \\ \frac{1}{\nu(|t|)} & \text{ if }& t\leq 0 .
             \end{array}
   \right.
\end{equation*}
In that case we say $\nu$ \textbf{induces} the growth rate $\mu$.
\item [e)] The map $p:\mathbb{R}^+_0\to [1,+\infty)$, $t\mapsto t+1$ induces a growth rate associated to the \textbf{nonuniform polynomial dichotomy} ($\NPD$). If moreover $\theta=\nu=0$, we obtain the \textbf{polynomial dichotomy} ($\PD$) \cite{Dragicevic3,Dragicevic5,Dragicevic6}.
\end{itemize}

\begin{definition}
\cite[p. 623]{Silva}  
    For a differentiable growth rate $\mu:\mathbb{R}\to \mathbb{R}$ we define the \textbf{nonuniform $\mu-$dichotomy spectrum} of (\ref{513}) by
    $$\Si_\NmuD(A):=\left\{\gamma\in\mathbb{R}: \dot{y}=\left[A(t)-\gamma\frac{\mu'(t)}{\mu(t)}\Id\right]y(t)\text { does not admit }\NmuD\right\}\,.$$

   Moreover, we call the complement of this set $\rho_\NmuD(A)=\mathbb{R}\setminus \Si_\NmuD(A)$ the \textbf{nonuniform $\mu$-resolvent set} of (\ref{513}).
\end{definition}

\begin{definition} \cite[p. 623]{Silva}  
    For a differentiable growth rate $\mu:\mathbb{R}\to \mathbb{R}$ we define the \textbf{uniform $\mu-$dichotomy spectrum} of (\ref{513}) by
    $$\Si_\muD(A):=\left\{\gamma\in\mathbb{R}: \dot{y}=\left[A(t)-\gamma\frac{\mu'(t)}{\mu(t)}\Id\right]y(t)\text { does not admit }\muD\right\}\,.$$

   Moreover, we call the complement of this set $\rho_\muD(A)=\mathbb{R}\setminus \Si_\muD(A)$ the \textbf{uniform $\mu$-resolvent set} of (\ref{513}).
\end{definition}

\begin{definition}
\cite[p. 630-631]{Silva}          We say the system (\ref{513}) has \textbf{nonuniform $\mu$-bounded growth rate with parameter $\epsilon>0$} or just \textbf{$\mu,\epsilon$-growth} if there are constants $K\geq 1$, $a\geq 0$ such that
    $$\norm{\Phi(t,s)}\leq K\left(\frac{\mu(t)}{\mu(s)}\right)^{\sgn(t-s)a}\mu(s)^{\sgn(s)\epsilon},\quad\forall\,t,s\in \mathbb{R}.$$

        Moreover, if we can choose $\epsilon=0$, we say the system has \textbf{uniform $\mu$-bounded growth} o just \textbf{$\mu$-growth}.
\end{definition}

A recent theorem by C. Silva states:

\begin{theorem}\cite[Theorem 8]{Silva}
    If (\ref{513}) has $\mu,\epsilon$-growth for some parameter $\epsilon>0$, then there exist some $n\in \{1,\dots,d\}$ such that its nonuniform $\mu$-dichotomy is nonempty, compact and has the form
        $$\Si_\NmuD(A)=\lambda_1\cup\cdots\cup \lambda_n,$$
        where each $\lambda_i=[a_i,b_i]$, with $a_i\leq b_i$ is an spectral interval.

\end{theorem}

Although the result is stated under the hypothesis of nonuniform $\mu$-bounded growth, if this condition is not obtained, the spectrum is still a finite (maybe empty) union of closed (maybe not compact) intervals, which follows easily from \cite[Lemma 7]{Silva}. We call each one of the open intervals that compose the resolvent set a {\it spectral gap}.

\smallskip
Moreover, although not stated in \cite{Silva}, these conclusions are also true for the uniform $\mu$-spectrum, {\it i.e.} $\Sigma_\muD(A)$ is also a finite (maybe empty) union of closed (maybe not compact) intervals, and the proof is the same as the author gives there. Furthermore, if the system admits uniform $\mu$-bounded growth, then $\Sigma_\muD(A)$ is nonempty and compact.

\smallskip

For the following, we consider a notion of kinematic similarity.

\begin{definition}
\cite[p. 636-637]{Silva}  Let $\epsilon\geq0$. We say (\ref{513}) a
\begin{equation}\label{514}
    \dot{z}=B(t)z,
\end{equation}
are \textbf{nonuniformly $(\mu,\epsilon)$-kinematically similar} if there is a differentiable matrix function $S:\mathbb{R}\to GL_d(\mathbb{R})$ and a constant $M_\epsilon>0$ such that
    \begin{equation}\label{515}
        \norm{S(t)}\leq M_\epsilon\mu(t)^{\sgn(t)\epsilon}\quad \text{ and }\quad\norm{S(t)^{-1}}\leq M_\epsilon\mu(t)^{\sgn(t)\epsilon},\qquad \forall\,t\in \mathbb{R}
    \end{equation}
    verifying that if $t\mapsto y(t)$ is solution to (\ref{513}), then $t\mapsto z(t)=S(t)^{-1}y(t)$ is solution of (\ref{514}) and analogously, if $t\mapsto z(t)$ is solution of (\ref{514}), then $t\mapsto y(t)=S(t)z(t)$ is solution of (\ref{513}). If $\epsilon=0$, we say the systems are \textbf{uniformly kinematically similar}.

    \smallskip
    Every $S$ satisfying (\ref{515}) for some $\epsilon\geq 0$ is called a \textbf{nonuniform Lyapunov matrix function with respect to $\mu$} and the change of variables $y(t)=S(t)x(t)$ is called a \textbf{nonuniform Lyapunov transformation with respect to $\mu$}.
\end{definition}

Following the steps of S. Siegmund \cite{Siegmund2}, the authors  J. Chu, F-F, Liao, Y. Xia, and W. Zhang \cite{Chu} gave a reducilibily and block diagonalization result through a kinematic similarity for nonuniform exponential dichotomies. Later, C. Silva \cite{Silva} extended these results for the nonuniform $\mu$-dichotomy. Silva's reducibility result \cite[Theorem 12]{Silva} for systems with $\mu$-dichotomies still holds for the uniform case, with the same demonstration as the author gives there. We state this result now, but slightly change the redaction in other to consider only the uniform case. 

\begin{theorem}\label{532}
    Assume system (\ref{513}) verifies uniform $\mu$-bounded growth. Then there exists some $n\in\{1,\dots,d\}$ such that
    $$\Sigma_\muD(A)=\lambda_1\cup\cdots\cup\lambda_n,$$
    where each $\lambda_i=[a_i,b_i]$ is a spectral interval. Moreover, there exists a (uniform) kinematic similarity between (\ref{513}) and a a block diagonal system 
    \begin{equation*}
         \dot{y}=B(t)y(t)=\begin{pmatrix}
B_1(t)&&\\
&\ddots&\\
&&B_{n}(t)
\end{pmatrix}y(t),
    \end{equation*}
    where $\Sigma_\muD(B_i)=\lambda_i$ for every $i=1,\dots,n$.
\end{theorem}

\subsection{The perturbation} Now we consider a {\it perturbation} of system (\ref{513}), i.e. a system of the form
\begin{equation}\label{516}
    \dot{x}=A(t)x+F(t,x),
\end{equation}
where $F$ is a $C^\ell$-Carath\'eodory class function verifying tailored conditions which we will express later.

\begin{definition}
    We say that a map $F:\Dom_F\subset \mathbb{R}\times \mathbb{R}^d\to \mathbb{R}^d$ is a \textbf{Carath\'eodory class} function if  for every interval  $I\subset \mathbb{R}$ and open set $U\subset \mathbb{R}^d$ such that $I\times U\subset \Dom_F$ it verifies:
    \begin{itemize}
        \item [i)] $F(t,\cdot):U\to \mathbb{R}^d$ is continuous for almost all fixed $t\in I$ (i.e. outside a set of zero Lebesgue measure),
        \item [ii)] $F(\cdot,x):I\to \mathbb{R}^d$ is measurable for all $x\in U$.
        \end{itemize}
Moreover, for $\ell\in \mathbb{N}$, we say $F$ is \textbf{$C^\ell$-Carath\'eodory class} if:
        \begin{itemize}
        \item [iii)] for almost $t\in I\subset \mathbb{R}$ and all $x\in U\subset \mathbb{R}^d$, the $\ell$-th partial derivative $D_2^\ell F(t,x)$ exists,
        \item [iv)] for each $j\in\{1,\dots,\ell\}$, the function $D_2^jF$ is Carath\'eodory class. 
    \end{itemize}
\end{definition}

It is well known that under these conditions and if $F(t,0)=0$ for all $t\in \mathbb{R}$, (\ref{516}) has uniquely defined solutions on some neighborhood of the origin, which is moreover a fixed point for this system.

\begin{definition}
Consider a growth rate $\mu$.    We say that a measurable function $\psi:\mathbb{R}\to \mathbb{R}^+_0$ is \textbf{$\mu$-admissible} if for every $\delta>0$ and every $t\in \mathbb{R}$ we have
    $$\int_{-\infty}^t\psi(s)\mu(s)^\delta ds+\int_{t}^\infty \psi(s)\mu(s)^{-\delta}ds<+\infty.$$

    If, furthermore
    $$\sup_{t\in \mathbb{R}}\left\{\int_{-\infty}^t\psi(s)\left(\frac{\mu(t)}{\mu(s)}\right)^{-\delta} ds+\int_{t}^\infty \psi(s)\left(\frac{\mu(t)}{\mu(s)}\right)^{\delta}ds\right\}<+\infty,$$
    we say that $\psi$ is \textbf{uniformly $\mu$-admissible}.
\end{definition}

For $\delta>0$, let us denote
$\zeta_{\psi,\mu,\delta}^+,\,\zeta_{\psi,\mu,\delta}^-:\mathbb{R}\to \mathbb{R}$ by
$$\zeta_{\psi,\mu,\delta}^+(t)=\int_t^\infty \psi(s)\mu(s)^{-\delta}ds\,\quad\text{
 and }\quad \zeta_{\psi,\mu,\delta}^-(t)=\int_{-\infty}^t \psi(s)\mu(s)^{\delta}ds.$$

Note that $\psi$ is admissible if and only if both $\zeta_\delta^+$ and $\zeta_\delta^-$ are continuous well defined functions for every $\delta>0$. Moreover, $\psi$ is uniformly admissible if and only if both the maps
$$t\mapsto \mu(t)^\delta\zeta_{\psi,\mu,\delta}^+(t)\, \quad\text{and}\quad t\mapsto\mu(t)^{-\delta}\zeta_{\psi,\mu,\delta}^{-}(t),$$
are bounded.

\begin{remark}
    {\rm A simple but useful observation is that both $\mu$-admissible and uniformly $\mu$-admissible functions define sets closed under addition and ponderation by positive constants.}
\end{remark}

\begin{example}
    {\rm
    For the exponential growth rate, {\it i.e.} $\mu(t)=e^{t}$, it is obvious that every bounded function is admissible (we may say, exponentially admissible), and actually uniformly exponentially admissible. Nevertheless, evidently this is not exhaustive, since, at least, every polynomial is exponentially admissible as well.
    }
\end{example}

\begin{example}\label{531}
{\rm    H\"ormander class functions are very important in the study of pseudo-differential operators, since they represent suitable symbols for them (\cite{Grushin,Kato}, among many others). The general definition of the H\"ormander spaces, for $m\in \mathbb{R}$ and $\rho,\delta\in [0,1]$, is the following:
$$S^m_{\rho,\delta}=\left\{f\in C^\infty(\mathbb{R}^{2d}): \forall\,\alpha,\beta\in \mathbb{N}^d\,, \exists\, c_{\alpha,\beta}\in \mathbb{R}\text{ s.t. }\norm{\frac{\partial^{|\alpha|+|\beta|}}{\partial x^\beta\partial \xi^\alpha}f(\xi,x)}\leq c_{\alpha,\beta}(1+|\xi|)^{m-\rho|\alpha|-\delta|\beta|}\right\}.$$

These spaces are known to be rather large classes of functions. In particular, each one is an infinite dimensional vector space.
Consider the scalar case {\it i.e.} $d=1$. Replace the variable $\xi=t$ and choose any derivative of any function in any of the spaces $S^{-1}_{\rho,\delta}$ of H\"ormander class functions. Call that function $\widetilde{\psi}$ and define $\psi:\mathbb{R}\to \mathbb{R}^+$ by $\psi(t)= \norm{\widetilde{\psi}(t,0)}$. The conclusion is that $\psi$ is polynomially admissible.
}
\end{example}

\subsection{Notions of equivalence}

\begin{definition}
    For a function $\xi:\mathbb{R}\to \mathbb{R}^+$ we define the \textbf{$\xi$-trumpet neighborhood} of the trivial solution by
    $$T_\xi=\left\{(t,x)\in \mathbb{R}\times \mathbb{R}^d: |x|\leq\xi (t)\right\}.$$

    When $\xi$ only takes values in some interval $[\varepsilon,\varepsilon']$ with $\varepsilon>0$, we say $T_\xi$ is a \textbf{tubular neighborhood} of the trivial solution.
\end{definition}

Note that if $d=2$ and if $\xi$ is a growth rate, then the set $T_\xi$ looks like an infinite solid trumpet on $\mathbb{R}^3$, which motivates this definition.  The term {tubular neighborhood} has been used several times in literature \cite{Cuong,Siegmund3}, among others.

\begin{definition}
Consider a second nonlinear system 
\begin{equation}\label{517}
    \dot{y}=\widetilde{F}(t,y),
\end{equation}
where $\widetilde{F}$ is a $C^\ell$-Carath\'eodory function with $\widetilde{F}(t,0)\equiv 0$ satisfying standard conditions of existence and uniqueness of solutions. We say the systems (\ref{517}) and (\ref{516}) are \textbf{nonuniformly locally $C^\ell$-equivalent} around the zero solution if there are two functions $\xi,\widetilde{\xi}:\mathbb{R}\to \mathbb{R}^+$ and maps
    $$H:T_\xi \to \mathbb{R}^d\quad\text{and}\quad H^{-1}:T_{\wxi}\to \mathbb{R}^d,$$
    such that:
    \begin{itemize}
        \item [i)] For every fixed $t\in \mathbb{R}$, $H(t,\cdot)$ and $H^{-1}(t,\cdot)$ are $C^\ell$-diffeomorphisms (or homeomorphisms if $\ell=0$) between $B_{\xi(t)}(0)$ (resp. $B_{\wxi(t)}(0)$) and its image, and inverses one of the other (when the definition domain allows it).
    \item [ii)] If $t\mapsto y(t)$ is a solution of (\ref{517}) whose graph is contained on $T_\wxi$, then $t\mapsto H^{-1}(t,y(t))$ is a solution of (\ref{516}). If $t\mapsto x(t)$ is solution of (\ref{516}) whose graph is contained on $T_\xi$, then $t\mapsto H(t,x(t))$ is solution of (\ref{517}).
    \item [iii)] Zero solutions are continuously mapped onto each other:
    $$\lim_{x\to 0}H(t,x)=\lim_{x\to 0}H^{-1}(t,x)=0,$$
    not necessarily uniformly for $t\in \mathbb{R}$.
    \end{itemize}

    If moreover both $\xi$ and $\wxi$ only take values on some interval $[\varepsilon,\varepsilon']$, with $\varepsilon>0$ and in {\rm iii)} the convergence is uniform, we say the the systems (\ref{517}) are (\ref{516}) are \textbf{uniformly locally $C^\ell$-equivalent} around the zero solution.
\end{definition}

\subsection{Our hypothesis}

\begin{itemize}
    \item [(H1)] The map $t\mapsto A(t)$ is locally integrable and has $\mu$-bounded growth. Moreover, there exist $n\in \mathbb{N}$, with $1\leq n \leq d$, $d_i\in \mathbb{N}$ such that $d_1+\cdots+d_n=d$, and $A_i:\mathbb{R}^d\to \mathcal{M}_{d_i}(\mathbb{R})$ maps such that 
    \begin{equation*}
    A(t)=\begin{pmatrix}
A_1(t)&&\\
&\ddots&\\
&&A_{n}(t)
\end{pmatrix}.
\end{equation*}

Furthermore, if $\Sigma_\muD(A)=\lambda_1\cup\cdots\cup\lambda_n$, where each $\lambda_i=[a_i,b_i]$ is a spectral interval, then $\Sigma_\muD(A_i)=\lambda_i$. We denote the evolution operator correspondent to the block $\dot{x}_i=A_i(t)x_i$ by $\Phi_i$.

\item [(H2)]  $F:\mathbb{R}\times \mathbb{R}^d\to \mathbb{R}^d$ is a $C^\ell$-Carath\'eodory map with $F=(F_1,\dots,F_n)$, where $F_i:\mathbb{R}\times \mathbb{R}^d\to \mathbb{R}^{d_i}$, verifying $F(t,0)=D_2F(t,0)=0$. Moreover, there is a $\mu$-admissible function $\psi$ such that 
$$\norm{D_2^mF(t,0)}\leq \psi(t),\quad\,\forall\,m=2,\dots,\ell\text{ and almost all }t\in \mathbb{R}.$$
\end{itemize}

\begin{example}
    {\rm
   Consider the polynomial growth rate, $d=1$ and a map $\varphi\in S^{-1}_{\rho,\delta}$. From Example \ref{531}, it is easily deduced that $F(t,x)=D_t\varphi(t,x)$ verifies (H2).
    }
\end{example}

We also consider later a uniform version of this hypothesis, {\it i.e.}
\begin{itemize}
\item [(H2')]  $F:\mathbb{R}\times \mathbb{R}^d\to \mathbb{R}^d$ is a $C^\ell$-Carath\'eodory map with $F=(F_1,\dots,F_n)$, where $F_i:\mathbb{R}\times \mathbb{R}^d\to \mathbb{R}^{d_i}$, verifying $F(t,0)=D_2F(t,0)=0$. Moreover, there is a \textbf{uniformly} $\mu$-admissible function $\psi$ such that 
$$\norm{D_2^mF(t,0)}\leq \psi(t),\quad\,\forall\,m=2,\dots,\ell\text{ and almost all }t\in \mathbb{R}.$$
\end{itemize}

\begin{remark}\label{533}
{\rm
Consider a system
$$\dot{x}=\mathfrak{A}(t)x+\mathfrak{F}(t,x),$$
where $\mathfrak{F}$ satisfies (H2) and $\mathfrak{A}$ verifies $\mu$-bounded growth but not (H1). Then, by Theorem \ref{532}, there is a uniform Lyapunov function $S:\mathbb{R}\to GL_d(\mathbb{R})$ which block diagonalizes $\mathfrak{A}$ while keeping its spectrum fixed. Moreover, defining $F(t,x)=S(t)^{-1}\mathfrak{F}(t,S(t)x)$, it is clear that $F$ still verifies (H2).
}
\end{remark}

Now we consider an algebraic structure for compact intervals. For $[a,b]$, $[c,d]$ and $\gamma\in \mathbb{R}$ we write
$$[a,b]+[c,d]:=[a+c,b+d]\qquad\text{y}\qquad\gamma\cdot [a,b]:=[\gamma a,\gamma b].$$

Consider now an index $j\in\{1,\dots,n\}$ and a multi index $\k=(k_1,\dots,k_n)\in \mathbb{N}_0^n$. If we write
\begin{equation}\label{518}
    \lambda_j\cap \sum_{i=1}^nk_i\lambda_i=\emptyset,
\end{equation}
we mean the interval $\lambda_j$ is disjoint to the compact interval obtained by the weighted by $\k$ sum of all spectral intervals. Equivalently, $\lambda_j$ is either at the left or the right to that sum, which means one of the following conditions verifies
\begin{equation}\label{519}
    a_j>k_1b_1+\cdots+k_nb_n,
\end{equation}
or
\begin{equation}\label{520}
  b_j<k_1a_1+\cdots+k_na_n ,
\end{equation}
in this case we can define
$$\dist\left(\lambda_j,\sum_{i=1}^nk_i\lambda_i\right)=\min\left\{a_j-k_1b_1+\cdots+k_nb_n,k_1a_1+\cdots+k_na_n-b_j\right\},$$
and it is clear that it is a strictly positive number. We call (\ref{518}) the \textbf{$\k$-th condition of nonresonance on position $j$}. With this notion we can define the last important hypothesis of this work:

\begin{itemize}
    \item [(H3)] \textit{(Spectral nonresonance)} The $\mu$-dichotomy spectrum $\Sigma_\NmuD(A)=\lambda_1\cup \cdots\cup\lambda_n$ of the linear system (\ref{513}) does not present resonances up until the degree $\ell$. That is, for every multi index $\k=(k_1,\dots,k_n)\in \mathbb{N}$ with $2\leq |\k|=k_1+\cdots+k_n\leq \ell$ and every position $j\in \{1,\dots,n\}$, the $\k$-th condition of nonresonance on position $j$ holds. In other words
    $$ \lambda_j\cap \sum_{i=1}^nk_i\lambda_i=\emptyset, \qquad\forall \,j\in\{1,\dots,n\}\text{ and all }2\leq \sum_{i=1}^nk_i\leq \ell.$$
\end{itemize}

\section{Elimination of nonresonant terms and normal forms}

This section is dedicated to presenting and demonstrating our main results. Although they have many parallels with S. Siegmund's Normal Forms Theorem \cite{Siegmund3}, these results are generalized for systems with $\mu$-dichotomies. For a multi index $\k=(k_1,\dots,k_n)$, we denote by $[x]^\k$ the tensor element in which the $\k$-th derivative ({\it i.e.} $D_2^\k=D_{2_1}^{k_1}\cdots D_{2_n}^{k_n}$) is applied as a linear transformation.

\begin{lemma}\label{509}

Suppose conditions (H1) and (H2) hold.  Suppose as well that the $\k$-th condition of nonresonance on position $j$ holds. Then, there is a nonuniform $C^\ell$-local equivalence between (\ref{516}) and
    \begin{equation*}
        \dot{x}=A(t)x+G(t,x),
    \end{equation*}
    where $G=(G_1,\dots,G_n)$ is a $C^\ell$-Carath\'eodory function that eliminates the $\k$-th Taylor term on position $j$ of $F$ and keeps all other Taylor terms fixed up to order $|\k|$, i.e., for any other multi index $\m\in \mathbb{N}_0^n$ with $1\leq |\m|\leq |\k|$:
    \begin{equation}\label{508}
        D^\m_2G_i(t,0)=\left\{ \begin{array}{lcc}
             D^\m_2F_i(t,0)&  \text{ for } &   \m\neq \k\,\, \quad\text{or}\quad i\neq j, \\
             \\ 0 & \text{ for } & \m=\k\quad\text{ and }\quad i= j.
             \end{array}
             \right.
    \end{equation}

\end{lemma}

\begin{proof}
We develop the proof in several steps. The main goal is to find a nonuniform $C^\ell$-local equivalence of the form $H(t,x)=x+h(t,x)$, where $h$ verifies some notion of smallness (which we express later).

\smallskip

\begin{itemize}
    \item \textit{Step 1: Definition and estimation of $h$.}

Let $\epsilon>0$. For each spectral interval $\lambda_i=[a_i,b_i]$ we choose two numbers $\a_i=\a_i(\epsilon)$ and $\b_i=\b_i(\epsilon)$  with
$$a_i-\epsilon\leq \a_i<a_i\qquad \text{and}\qquad b_i<\b_i\leq b_i+\epsilon.$$

\smallskip

An immediate consequence of \cite[Lemma 7]{Silva} is that for any system $\dot{x}=B(t)x$ and $\gamma\in \mathbb{R}$ such that $\gamma>\wgamma$ for every $\wgamma\in \Sigma_\muD(B)$, the system
$$\dot{x}=\left[B(t)-\gamma\frac{\mu'(t)}{\mu(t)}\Id\right]x(t),$$
admits $\muD$ with projector identity. This follows from the fact that the rank of the projector associated to the dichotomy is strictly increasing between different spectral gaps. Analogously, if $\gamma<\wgamma$ for every $\wgamma\in \Sigma_\muD(B)$, then the $\gamma$-shifted system
admits $\muD$ with zero projector. 

\smallskip

Then, as $\hat{a}_i<\Sigma_\muD(A_i)<\hat{b}_i$, there are constants $\alpha_i<0$, $\beta_i>0$ and $K\geq 1$ such that
\begin{align*}
    \norm{\Phi_i(t,s)\left(\frac{\mu(t)}{\mu(s)}\right)^{-\b_i}}\leq K\left(\frac{\mu(t)}{\mu(s)}\right)^{\alpha_i}\leq K&\,,\quad\forall\, t\geq s,\\
    \norm{\Phi_i(t,s)\left(\frac{\mu(t)}{\mu(s)}\right)^{-\a_i}}\leq K\left(\frac{\mu(t)}{\mu(s)}\right)^{\beta_i}\leq K&\,\quad \forall\,t\leq s,
\end{align*}
thus
\begin{align*}
    \norm{\Phi_i(t,s)}\leq K\left(\frac{\mu(t)}{\mu(s)}\right)^{\b_i}\leq K\left(\frac{\mu(t)}{\mu(s)}\right)^{b_i+\epsilon}&\,,\quad \forall\,t\geq s,\\
    \norm{\Phi_i(t,s)}\leq K\left(\frac{\mu(t)}{\mu(s)}\right)^{\a_i}\leq K\left(\frac{\mu(t)}{\mu(s)}\right)^{a_i-\epsilon}&\,,\quad \forall\,t\leq s.
\end{align*}

Now, for all $t\in \mathbb{R}$, $x\in \mathbb{R}^d$ and almost all $s\in \mathbb{R}$ the following estimation follows:

\begin{align*}
    \norm{\Phi_j(t,s)\frac{1}{\k!}D_x^{\k}F_j(s,0)\cdot[\Phi(s,t)x]^{\k}}=\norm{\Phi_j(t,s)\frac{1}{\k!}D_x^{\k}F_j(s,0)\cdot[\Phi_1(s,t)x_1]^{k_1}\cdots [\Phi_n(s,t)x_n]^{k_n}}\\
   \leq \norm{\Phi_j(t,s)}\cdot\frac{1}{\k!}\cdot\norm{D_x^{\k}F_j(s,0)}\cdot\norm{\Phi_1(s,t)}^{k_1}\cdot\norm{x_1}^{k_1}\cdots \norm{\Phi_n(s,t)}^{k_n}\cdot\norm{x_n}^{k_n}\\
   \leq \frac{\psi(s)}{\k!}\cdot\norm{x_1}^{k_1}\cdots \norm{x_n}^{k_n}\cdot\norm{\Phi_j(t,s)}\cdot\norm{\Phi_1(s,t)}^{k_1}\cdots\norm{\Phi_n(s,t)}^{k_n}.
\end{align*}

Thus, for $t\leq s$ we have
\begin{align*}
\norm{\Phi_j(t,s)\frac{1}{\k!}D_x^{\k}F_j(s,0)\cdot[\Phi(s,t)x]^{\k}}\leq&\frac{\psi(s)  K^{|\k|+1}}{\k!}\norm{x_1}^{k_1}\cdots \norm{x_n}^{k_n}\\
&\cdot\left(\frac{\mu(t)}{\mu(s)}\right)^{a_j-\left[k_1b_1+\cdots+ k_nb_n\right]-(|\k|+1)\epsilon},
\end{align*}
hence, if (\ref{519}) holds and we choose $\epsilon$ to be small enough, we obtain for $t\leq s$
$$\norm{\Phi_j(t,s)\frac{1}{\k!}D_x^{\k}F_j(s,0)\cdot[\Phi(s,t)x]^{\k}}\leq \frac{\psi(s)  K^{|\k|+1}}{\k!}\norm{x_1}^{k_1}\cdots \norm{x_n}^{k_n} \left(\frac{\mu(t)}{\mu(s)}\right)^{\frac{1}{2}\dist\left(\lambda_j,\sum_{i=1}^nk_i\lambda_i\right)}.$$

On the other hand, for $t\geq s$ we have
\begin{align*}
    \norm{\Phi_j(t,s)\frac{1}{\k!}D_x^{\k}F_j(s,0)\cdot[\Phi(s,t)x]^{\k}}\leq&\frac{\psi(s)  K^{|\k|+1}}{\k!}\norm{x_1}^{k_1}\cdots \norm{x_n}^{k_n}\\
&\cdot\left(\frac{\mu(t)}{\mu(s)}\right)^{b_j-\left[k_1a_1+\cdots+ k_na_n\right]+(|\k|+1)\epsilon},
\end{align*}
hence, if (\ref{520}) holds and we choose $\epsilon$ to be small enough, we obtain for $t\geq s$
$$\norm{\Phi_j(t,s)\frac{1}{\k!}D_x^{\k}F_j(s,0)\cdot[\Phi(s,t)x]^{\k}}\leq \frac{\psi(s) K^{|\k|+1}}{\k!}\norm{x_1}^{k_1}\cdots \norm{x_n}^{k_n}  \left(\frac{\mu(t)}{\mu(s)}\right)^{-\frac{1}{2}\dist\left(\lambda_j,\sum_{i=1}^nk_i\lambda_i\right)}.$$

\smallskip

Now, define the map $h=(h_1,\dots,h_n):\mathbb{R}\times\mathbb{R}^d\to \mathbb{R}^d$ given by
    \begin{equation*}
 h_i(t,x)=\left\{ \begin{array}{lcc}
             0&  \text{if} &  i\neq j,\\
             \\
               \mathlarger{\int}_t^{\infty}\Phi_j(t,s)\frac{1}{\k!}D_x^{\k}F_j(s,0)\cdot[\Phi(s,t)x]^{\k}ds&  \text{ if } &  i= j \text{  and (\ref{519}) holds},\\
             \\ -\mathlarger{\int}_{-\infty}^{t}\Phi_j(t,s)\frac{1}{\k!}D_x^{\k}F_j(s,0)\cdot[\Phi(s,t)x]^{\k}ds&  \text{ if } &  i= j \text{ and (\ref{520}) holds}.
             \end{array}
             \right.
\end{equation*}

For simplicity in notation, let us write $\dist(j,\k)=\dist\left(\lambda_j,\sum_{i=1}^nk_i\lambda_i\right)$. By the previous estimations we obtain that $h$ is well defined and
    \begin{equation*}
 \norm{h(t,x)}\leq\left\{ \begin{array}{lcc}
            
             \dfrac{K^{|\k|+1}}{\k!}  \mu(t)^{\frac{1}{2}\dist(j,\k)}\zeta_{\psi,\mu,\frac{1}{2}\dist(j,\k)}^+(t)\norm{x_1}^{k_1}\cdots\norm{x_n}^{k_n}&  \text{ if  (\ref{519}) holds},\\
             \\\dfrac{K^{|\k|+1}}{\k!}\mu(t)^{-\frac{1}{2}\dist(j,\k)}\zeta_{\psi,\mu,\frac{1}{2}\dist(j,\k)}^-(t)\norm{x_1}^{k_1}\cdots\norm{x_n}^{k_n}&  \text{ if  (\ref{520}) holds}.
             \end{array}
             \right.
\end{equation*}

We will continue the proof assuming (\ref{519}), but the other case follows similarly. In order to simplify notations, let us write
\begin{equation}\label{510}
    \zeta(t)=  \mu(t)^{\frac{1}{2}\dist(j,\k)}\zeta_{\psi,\mu,\frac{1}{2}\dist(j,\k)}^+(t),
\end{equation}
and note that it is a continuous strictly positive function. The conclusion is
\begin{equation}\label{501}
    \norm{h(t,x)}\leq \dfrac{K^{|\k|+1}}{\k!}\zeta(t)\norm{x_1}^{k_1}\cdots\norm{x_n}^{k_n}\leq\dfrac{K^{|\k|+1}}{\k!} \zeta(t)\norm{x}^{|\k|}\,.
\end{equation}

\item \textit{Step 2: The map $H:\mathbb{R}\times \mathbb{R}^d\to \mathbb{R}^d$, given by $(t,x)\mapsto x+h(t,x)$, is continuous (in both variables simultaneously) and infinitely continuously differentiable respect to $x$.}

\smallskip

It is trivial that if $h$ verifies the statement, then $H$ does too. For the continuity of $(t,x)\mapsto h(t,x)$, note first that each $t\mapsto \Phi_i(t,s)$ and $t\mapsto \Phi_i(s,t)$ are continuous. Then, the map
\begin{equation}\label{521}
    (t,x)=(t,x_1,\dots,x_n)\mapsto [\Phi_1(s,t)x_1]^{k_1}\cdots [\Phi_n(s,t)x_n]^{k_n},
\end{equation}
is continuous as well. Thus, as $h$ verifies the estimation (\ref{501}), and $\zeta$ is continuous, we obtain that $h$ is locally (in $t$) bounded. Now, Lebesgue's Dominated Convergence Theorem ensures the continuity of $(t,x)\mapsto h(t,x)$.

\smallskip

Similarly, as $x$ only appears on the integrand on $h$, on a polynomial transformation of the application (\ref{521}), it is clearly infinitely times differentiable respect to $x$. Moreover, once again Lebesgue's Dominated Convergence Theorem implies that all these derivatives respect to $x$ are continuous functions of  $(t,x)$ (both variables simultaneously).

\smallskip

Furthermore, we have an explicit formula for its first derivative (evaluated, as a linear transformation, on some vector $\varsigma\in \mathbb{R}^d$). 
\begin{align}\label{522}
 D_2h_j(t,x)\cdot \varsigma=&   \int_t^{\infty}D_x\left[ \Phi_j(t,s)\frac{1}{\k!}D_2^{\k}F_j(s,0)\cdot[\Phi_1(s,t)x_1]^{k_1}\cdots[\Phi_n(s,t)x_1]^{k_n}\right]\cdot \xi \,ds \nonumber \\
 \nonumber \\ 
 =&\sum_{i=1,\dots,n\,:\,k_i\geq 1}k_i\int_{t}^{\infty}\Phi_j(t,s)\frac{1}{\k!}D^{\k}_2F(s,0)[\Phi_1(s,t)x_1]^{k_1}\cdots\\
   \nonumber \\ 
 &\cdots[\Phi_i(s,t)\varsigma_i]\cdot[\Phi_i(s,t)x_i]^{k_i-1}\cdots[\Phi_n(s,t)x_1]^{k_n}\,ds. \nonumber
\end{align}

From this it is easily deduced that  $D_2h(t,0)=0$ for every $t\in \mathbb{R}$. Consider a multi index  $\m=(m_1,\dots,m_n)$. If  $m_i>k_i$ for some $i=1,\dots,n$, then $D_{2_i}^{m_i}h_j(t,x)\equiv 0$. Hence, by commutativity of the differential operators we obtain $D_2^\m h_j(t,x)\equiv 0$.

\smallskip

On the other hand, if $m_i<k_i$ for every $i=1,\dots,n$, then $D_2^\m h_j(t,0)= 0$. In other words, for any multi index $\m$ different than $\k$, we have $D_2^\m h_j(t,0)=0$.

\smallskip

\item \textit{Step 3: The partial derivative $D_1H(t,x)=D_1h(t,x)$ exists for every $x\in \mathbb{R}^d$ an almost all $t\in \mathbb{R}$. Moreover,  $D_1H:\mathbb{R}\times \mathbb{R}^d\to \mathbb{R}^d$ is a $C^\ell$-Carath\'eodory class function.}

\smallskip

The map $(t,x)\mapsto h_j(t,x)$ is derivable respect to $t$ on almost every $t\in \mathbb{R}$ and all $x\in \mathbb{R}^d$, since it is defined through a (locally) bounded integral. Indeed
\begin{align*}
    D_1h_j(t,x)=&-\Phi_j(t,t)\frac{1}{\k!}D_2^{\k}F_j(t,0)\cdot[\Phi_1(t,t)x_1]^{k_1}\cdots [\Phi_n(t,t)x_n]^{k_n}\\
    \\&+\int_t^{\infty}D_t\left[ \Phi_j(t,s)\frac{1}{\k!}D_2^{\k}F_j(s,0)\cdot[\Phi_1(s,t)x_1]^{k_1}\cdots[\Phi_n(s,t)x_1]^{k_n}\right]ds.
\end{align*}

Thus, using the identity $D_t\Phi_i(s,t)=-\Phi_i(s,t)A_i(t)$ we obtain
\begin{align}\label{523}
    D_1h_j(t,x)=&-\frac{1}{\k!}D_2^{\k}F_j(t,0)\cdot [x]^{\k}+A_j(t)h_j(t,x)\nonumber \\
    \nonumber \\&+\sum_{i=1,\dots,n\,:\,k_i\geq 1}k_i\int_{t}^{\infty}\Phi_j(t,s)\frac{1}{\k!}D^{\k}_2F(s,0)[\Phi_1(s,t)x_1]^{k_1}\cdots\\
    \nonumber\\&\cdots[-\Phi_i(s,t)A_i(t)x_i]\cdot[\Phi_i(s,t)x_i]^{k_i-1}\cdots[\Phi_n(s,t)x_1]^{k_n}\,ds,\nonumber
\end{align}
from where, as $h$ is infinitely times continuously differentiable respect to $x$, it is clear that the first two addends are $C^\ell-$Carath\'eodory. The differentiability of the third added respect to $x$ is analogous to that of $h$. 

\smallskip

Hence, $D_1H$ is a $C^\ell$-Carath\'eodory, since for every fixed $t$, every derivative respect to $x$ is continuous, and for every fixed $x$, said derivative is a measurable function of $t$ (recall that $t\mapsto A_i(t)$ is measurable, but might be discontinuous).

\smallskip

\item \textit{Step 4: There is a trumpet neighborhood of the trivial solution where for every $t\in \mathbb{R}$ the map $H(t,\cdot)$ defines a $C^\ell$-diffeomorphism. This diffeomorphism maps a trumpet neighborhood of the trivial solution into another trumpet neighborhood.}

\smallskip

It is enough to note that (\ref{522}) implies
\begin{equation}\label{502}
    \norm{D_2h(t,x)}\leq nK^{|\k|+1}\zeta(t)\norm{x}^{|\k|-1}.
\end{equation}

Define
$\xi(t)=\left(2K^{|\k|+1}n\zeta(t)\right)^{\frac{1}{1-|\k|}}$. Then, of the $\xi$-trumpet, we have $\norm{D_2h(t,x)}\leq \frac{1}{2}$. This implies that $\Id+D_2h(t,x)=D_2H(t,x)$ is invertible on the $\xi$-trumpet. By the inverse function Theorem the statement follows.

\smallskip

We define 
$$H(T_\xi)=\{H(t,x):(t,x)\in T_\xi\}=\{H(t,x):\norm{x}\leq \xi(t)\}.$$ 

Clearly, there is some $\wxi:\mathbb{R}\to \mathbb{R}^+$ such that $H(T_\xi)=T_\wxi$.  Moreover, as $H$ and $\xi$ are continuous, $\wxi$ is continuous as well.

\smallskip

\item \textit{Step 5: On these domains, for every fixed $t\in \mathbb{R}$, the maps $H(t,\cdot)$ and $H^{-1}(t,\cdot)$ are Lipschitz.}

\smallskip

From  (\ref{502}), it follows that for $x,\overline{x}\in \mathbb{R}^d$ with $\norm{x},\norm{\overline{x}}\leq \xi(t)$, we have
\begin{equation}\label{504}
  \norm{h(t,x)-h(t,\overline{x})}\leq\frac{1}{2} \norm{x-\overline{x}},  
\end{equation}
thus
$$\norm{H(t,x)-H(t,\overline{x})}\leq\frac{3}{2}\norm{x-\overline{x}}.$$

Now, for $y,\overline{y}\in \mathbb{R}^d$ with $\norm{y},\norm{\overline{y}}\leq \xi(t)$, we have
\begin{align*}
    \frac{1}{2}\norm{y-\overline{y}}=\norm{y-\overline{y}}-\frac{1}{2}\norm{y-\overline{y}}\leq \norm{y-\overline{y}}-\norm{h(t,y)-h(t,\overline{y})}\leq \norm{H(t,y)-H(t,\overline{y})},
\end{align*}
thus, for $x,\overline{x}\in T_\wxi$
\begin{equation}\label{503}
   \norm{H^{-1}(t,x)-H^{-1}(t,\overline{x})}\leq 2\norm{x-\overline{x}}. 
\end{equation}

\item  {\it Step 6: The map $H^{-1}:T_\wxi\subset \mathbb{R}\times \mathbb{R}^d\to \mathbb{R}^d$ is continuous (in both variables simultaneously) and $\ell$ times continuously differentiable respect to $x$.}

\smallskip

From (\ref{503}) we have
\begin{align*}
    \norm{H^{-1}(t,x)-H^{-1}(t_0,x_0)}&\leq \norm{H^{-1}(t,x)-H^{-1}(t,x_0)}+\norm{H^{-1}(t,x_0)-H^{-1}(t_0,x_0)}\\
    &\leq 2\norm{x-x_0}+\norm{H^{-1}(t,x_0)-H^{-1}(t_0,x_0)},
\end{align*}
thus, it is enough to prove that $\lim_{t\to t_0}H^{-1}(t,x_0)=H^{-1}(t_0,x_0)$. Consider the identity
\begin{equation}\label{524}
x_0=H(t,H^{-1}(t,x_0))=H^{-1}(t,x_0)+h(t,H^{-1}(t,x_0)),
\end{equation}
from where
\begin{equation}\label{525}
    H^{-1}(t,x_0)=x_0-h(t,H^{-1}(t,x_0)),
\end{equation}
hence, using (\ref{504}) we obtain
\begin{eqnarray*}
    \norm{H^{-1}(t,x_0)-H^{-1}(t_0,x_0)}&=&\norm{h(t,H^{-1}(t,x_0))-h(t_0,H^{-1}(t_0,x_0))}\\
    &\leq&\norm{h(t,H^{-1}(t,x_0))-h(t,H^{-1}(t_0,x_0))}\\
    &&+\norm{h(t,H^{-1}(t_0,x_0))-h(t_0,H^{-1}(t_0,x_0))}\\
    &\leq&\frac{1}{2}\norm{H^{-1}(t,x_0)-H^{-1}(t_0,x_0)}\\
    &&+\norm{h(t,H^{-1}(t_0,x_0))-h(t_0,H^{-1}(t_0,x_0))},
\end{eqnarray*}
from where, reorganizing terms we have 
$$\norm{H^{-1}(t,x_0)-H^{-1}(t_0,x_0)}\leq 2\,\norm{h(t,H^{-1}(t_0,x_0))-h(t_0,H^{-1}(t_0,x_0))},$$
thus, the continuity of $H^{-1}$ is deduced from the continuity of $h$ (Step 2).

\smallskip

Finally, as for every fixed $t$ the maps $H(t,\cdot)$ and $H^{-1}(t,\cdot)$ define $C^\ell$-diffeomorphisms (Step 4), then $H^{-1}$ is $\ell$ times differentiable respect to $x$. Consider a multi index $\m$ with $|\m|\leq \ell$. To conclude the continuity of  $(t,x)\mapsto D^\m_2H^{-1}(t,x)$ it is enough to see that it can be written as compositions and operations of $H^{-1}$, the inversion of linear transformations, and $D_2^{\widetilde{\m}} H$, with $|\widetilde{\m}|\leq |\m|$, all of them being continuous functions.

\smallskip

\item \textit{Step 7: If $\wmu$ is a solution of (\ref{516}) which lies on $T_\xi$, then $H(\cdot,\wmu(\cdot))$ is a solution of $\dot{x}=\widetilde{G}(t,x)$, where $\widetilde{G}:T_\xi\subset \mathbb{R}\times \mathbb{R}^d\to \mathbb{R}^d$ is given by}
\begin{eqnarray*}
    \widetilde{G}(s,x)&=&D_1H\left(s,H^{^{-1}}(s,x)\right)\\
    &&+D_2H\left(s,H^{-1}(s,x)\right)\cdot \left[A(s)H^{-1}(s,x)+F(s,H^{-1}(s,x))\right].
\end{eqnarray*}

Note that we can define such $\widetilde{G}$ because of Steps 2 and 3. Consider the function $s\mapsto \wnu(s):=H(s,\wmu(s))$. As $\wmu(s)=H^{-1}\left(s,H(s,\wmu(s))\right)$, then $\wmu(s)=H^{-1}(s,\nu(s))$. We have
\begin{eqnarray*}
    D_s[\wnu(s)]&=&D_s\left[H(s,\wmu(s))\right]\\
    &=&D_1H(s,\wmu(s))+D_2H(s,\wmu(s))\cdot D_s[\wmu(s)]\\
&=&D_1H(s,\wmu(s))+D_2H(s,\wmu(s))\cdot \left[A(s)\wmu(s)+F(s,\wmu(s))\right]\\
&=&D_1H\left(s,H^{-1}(s,\wnu(s))\right)\\
&&+D_2H\left(s,H^{-1}(s,\wnu(s))\right)\cdot \left[A(s)H^{-1}(s,\wnu(s))+F\left(s,H^{-1}(s,\wnu(s))\right)\right]\\
&=&\widetilde{G}(s,\wnu(s)).
\end{eqnarray*}

\item \textit{Step 8: If $\wnu$ is a solution of $\dot{x}=\widetilde{G}(t,x)$  which lies on $T_\wxi$, then $H^{-1}(\cdot,\wnu(\cdot))$ is a solution of (\ref{516}).}

\smallskip

Consider the function $s\mapsto \wmu(s):=H^{-1}(s,\wnu(s))$. As $\wnu(s)=H\left(s,H^{-1}(s,\wnu(s))\right)$, then $\wnu(s)=H(s,\wmu(s))$. We have
\begin{eqnarray*}
    D_s[\wmu(s)
]&=&D_s\left[H^{-1}(s,\wnu(s))\right]\\
   &=&D_1H^{-1}(s,\wnu(s))+D_2H^{-1}(s,\wnu(s))\cdot D_s[\wnu(s)]\\
    &=&D_1H^{-1}(s,\wnu(s))+D_2H^{-1}(s,\wnu(s))\cdot \widetilde{G}(s,\wnu(s))\\
    &=&D_1H^{-1}(s,\wnu(s))+D_2H^{-1}(s,\wnu(s))\cdot D_1H\left(s,H^{-1}(s,\wnu(s))\right)\\
    &&+D_2H^{-1}(s,\wnu(s))\cdot D_2H\left(s,H^{-1}(s,\wnu(s))\right) \\
    &&\cdot \left[A(s)H^{-1}(s,\wnu(s))+F\left(s,H^{-1}(s,\wnu(s))\right)\right]\\
    &=&D_1H^{-1}\left(s,H(s,\wmu(s))\right)+D_2H^{-1}\left(s,H(s,\wmu(s))\right)\cdot D_1H\left(s,\wmu(s)\right)\\
    &&+D_2H^{-1}(s,H(s,\wmu(s)))\cdot D_2H\left(s,\wmu(s))\right)\cdot \left[A(s)\wmu(s)+F\left(s,\wmu(s)\right)\right]\\
    &=&\left.D_t\left[H^{-1}(t,H(t,x))\right]\right|_{t=s,\,x=\wmu(s)}\\
    &&+\left.D_x\left[H^{-1}(s,H(s,x))\right]\right|_{x=\wmu(s)}\cdot \left[A(s)\wmu(s)+F\left(s,\wmu(s)\right)\right]\\
    &=&A(s)\wmu(s)+F\left(s,\wmu(s)\right).
\end{eqnarray*}

\smallskip

\item \textit{Step 9: The map $\widetilde{G}$ is $C^\ell$-Carath\'eodory. Moreover, there are maps $R_i:\Dom_{R_i}\subset \mathbb{R}\times \mathbb{R}^d\to \mathbb{R}^{d_i}$, with $R_i(t,0)=0$, such that
$$ \widetilde{G}_j(t,x)=A_j(t)x_j+F_j(t,x)-\frac{1}{\k!}D_2^{\k}F_j(t,0)\cdot [x]^{\k}+R_j(t,x),$$
and
$$\widetilde{G}_i(t,x)=A_i(t)x_i+F_i(t,x)+R_i(t,x), \qquad\forall\,i\neq j.$$
}

\smallskip

The first statement is a simple observation which follows from $D_1H$ being of $C^\ell$-Carath\'eodory class (Step 3), the maps $H$ and $H^{-1}$ are $\ell$ times continuously differentiable respect to $x$ (Steps 2 and 6), $F$ is $C^\ell$-Carath\'eodory and $A$ is measurable (by hypothesis). With this, and knowing that composing Carath\'eodory functions with measurable functions preserves measurability, the statement is obtained.

\smallskip
Using the identities (\ref{522}) and (\ref{523}) we obtain
$$D_1h_j(t,x)=-\frac{1}{\k!}D_2^{\k}F_j(t,0)\cdot [x]^{\k}+A_j(t)h_j(t,x)-D_2h_j(t,x)\cdot[A(t)x].$$

Now, the $j$-th component of $\widetilde{G}$ is
\begin{eqnarray*}
    \widetilde{G}_j(t,x)&=&D_1H_j\left(t,H^{^{-1}}(t,x)\right)\\
    &&+D_2H_j\left(t,H^{-1}(t,x)\right)\cdot \left[A(t)H^{-1}(t,x)+F(t,H^{-1}(t,x))\right],
\end{eqnarray*}
thus, combining these identities we have
\begin{eqnarray*}
    \widetilde{G}_j(t,x)&=&-\frac{1}{\k!}D_2^{\k}F_j(t,0)\cdot \left[H^{-1}(t,x)\right]^{\k}+A_j(t)h_j\left(t,H^{-1}(t,x)\right)\\
    &&-D_2h_j\left(t,H^{-1}(t,x)\right)\cdot\left[A(t)H^{-1}(t,x)\right]\\
    &&+\left[\Id_{\mathbb{R}^{d_j}}+ D_2h_j(t,H^{-1}(t,x))\right]\cdot \left[A(t)H^{-1}(t,x)+F(t,H^{-1}(t,x))\right]\\
     &=&-\frac{1}{\k!}D_2^{\k}F_j(t,0)\cdot \left[H^{-1}(t,x)\right]^{\k}+A_j(t)h_j\left(t,H^{-1}(t,x)\right)\\
    &&+D_2h_j\left(t,H^{-1}(t,x)\right)\cdot F(t,H^{-1}(t,x))\\
    &&+A_j(t)H_j^{-1}(t,x)+F_j(t,H^{-1}(t,x))\\
    &=&A_j(t)\left[H_j^{-1}(t,x)+h_j\left(t,H^{-1}(t,x)\right)\right]-\frac{1}{\k!}D_2^{\k}F_j(t,0)\cdot \left[H^{-1}(t,x)\right]^{\k}\\
    &&+D_2h_j\left(t,H^{-1}(t,x)\right)\cdot F(t,H^{-1}(t,x))+F_j(t,H^{-1}(t,x)),
\end{eqnarray*}
hence, using (\ref{524}) and defining $R_j:\Dom_{R_j}\subset\mathbb{R}\times\mathbb{R}^d\to \mathbb{R}^{d_j} $ by
\begin{align*}
  R_j(t,x)=&  D_2h_j\left(t,H^{-1}(t,x)\right)\cdot F(t,H^{-1}(t,x))+\frac{1}{\k!}D_2^{\k}F_j(t,0)\cdot \left([x]^\k-\left[H^{-1}(t,x)\right]^{\k}\right)\\
  &+F_j(t,H^{-1}(t,x))-F_j(t,x),
\end{align*}
we obtain
\begin{equation*} 
    \widetilde{G}_j(t,x)=A_j(t)x_j+F_j(t,x)-\frac{1}{\k!}D_2^{\k}F_j(t,0)\cdot [x]^{\k}+R_j(t,x),
\end{equation*}
and $R_j(t,0)=0$.  Now, for other index $i\neq j$, we have $H_i(t,x)=x_i$, thus $D_1H_i(t,x)\equiv 0$ and $D_2H_i(t,x)=\Id_{\mathbb{R}^{d_i}}$, hence
\begin{eqnarray*}
    \widetilde{G}_i(t,x)&=&D_1H_i\left(t,H^{^{-1}}(t,x)\right) \nonumber\\
    &&+D_2H_i\left(t,H^{-1}(t,x)\right)\cdot \left[A(t)H^{-1}(t,x)+F(t,H^{-1}(t,x))\right]\nonumber \\
    &=&A_i(t)H_i^{-1}(t,x)+F_i(t,H^{-1}(t,x))\nonumber\\
&=&A_i(t)x_i+F_i(t,H^{-1}(t,x)),
\end{eqnarray*}
thus, defining $R_i:\Dom_{R_i}\subset\mathbb{R}\times\mathbb{R}^d\to \mathbb{R}^{d_i}$ by
$$R_i(t,x)=F_i(t,H^{-1}(t,x))-F_i(t,x),$$
we have $R_i(t,0)=0$ and
$$\widetilde{G}_i(t,x)=A_i(t)x_i+F_i(t,x)+R_i(t,x), \qquad\forall\,i\neq j.$$

\item {\it Step 10: Define $R=(R_1,\dots, R_n):\Dom_R\subset \mathbb{R}\times \mathbb{R}^d\to \mathbb{R}^d$. For every multi index $\m$ with $|\m|\leq |\k|$ we have $D_2^\m R(t,0)=0$ for almost all $t\in \mathbb{R}$.}

\smallskip

From (\ref{502}) we have
$$\norm{D_2h_j\left(t,H^{-1}(t,x)\right)}\leq nK^{|\k|+1}\zeta(t)\norm{H^{-1}(t,x)}^{|\k|-1}.$$

On the other hand, as $F(t,0)=0$ and $D_2F(t,0)=0$, there is a map $\widetilde{c}:\mathbb{R}\to \mathbb{R}^+$ such that for small enough $x$ we have
$$\norm{F(t,x)}\leq \widetilde{c}(t)\norm{x}^2,$$
for almost all $t\in \mathbb{R}$, with the only possible exceptions where $D_2^kF(t,x)$ is not defined.

\smallskip

From (\ref{503}) we have $\norm{H^{-1}(t,x)}\leq 2\norm{x}$. In conclusion
$$\norm{D_2h_j\left(t,H^{-1}(t,x)\right)\cdot F(t,H^{-1}(t,x))}\leq nK^{|\k|+1}\zeta(t)\,\widetilde{c}(t)\,2^{|\k|+1}\norm{x}^{|\k|+1},$$
hence
\begin{equation}\label{507}
    \lim_{x\to 0}\frac{\norm{D_2h_j\left(t,H^{-1}(t,x)\right)\cdot F(t,H^{-1}(t,x))}}{\norm{x}^{|\k|}}=0.
\end{equation}

Similarly, from (\ref{501}) we have
\begin{equation}\label{505}
    \norm{h(t,H^{-1}(t,x))}\leq \frac{K^{|\k|+1}}{\k!}\zeta(t)\,2^{|\k|+1}\norm{x}^{|\k|},
\end{equation}
which in conjunction with identity (\ref{525}) implies
\begin{equation}\label{506}
    \lim_{x\to 0}\frac{1}{\norm{x}^{|k|}}\norm{\frac{1}{\k!}D_2^{\k}F_j(t,0)\cdot \left([x]^\k-\left[H^{-1}(t,x)\right]^{\k}\right)}=0.
\end{equation}

Now, from (\ref{505}) we have that all derivatives of the map $x\mapsto h(t,H^{-1}(t,x))$ are zero in the origin, up until the order $|\k|-1$.  Once again from (\ref{525}) this implies
$$D_2^mH^{-1}(t,0)=0,\quad\forall\,m=2,\dots, |\k|-1,$$
which in conjunction with $D_2H^{-1}(t,0)=\Id$ and $D_2F(t,0)=0$ implies 
$$\left.D_x^m\left[F_j(t,H^{-1}(t,x))-F_j(t,x)\right]\right|_{x=0}=0,\quad\forall\,m=1,\dots, |\k|,$$
which with (\ref{507}) and (\ref{506}) complete the demonstration of the step.

\smallskip

\item {\it Step 11: Definition of the map $G$ and verification of (\ref{508}).}

\smallskip

It is enough to define $G=(G_1,\dots,G_n):\Dom_G\subset \mathbb{R}\times \mathbb{R}^d\to \mathbb{R}^d$  by
$$G_j(t,x)=F_j(t,x)-\frac{1}{\k!}D_2^{\k}F_j(t,0)\cdot [x]^{\k}+R_j(t,x),$$
and
$$G_i(t,x)=F_i(t,x)+R_i(t,x), \qquad\forall\,i\neq j.$$

Then, by Steps 9, 10 and the simple observation that the term  $\frac{1}{\k!}D_2^{\k}F_j(t,0)\cdot [x]^{\k}$ eliminates exactly the $\k$-th Taylor term on position $j$, the lemma follows.
\end{itemize}
\end{proof} 

A direct consequence of the preceding lemma is achieved by merely iterating it. 

\begin{corollary}
    Suppose conditions (H1) and (H2) hold.  Suppose as well that all conditions of nonresonance on every position hold for a certain order $k$. Then, there is a nonuniform $C^\ell$-local equivalence between (\ref{516}) and
    \begin{equation*}
        \dot{x}=A(t)x+G(t,x),
    \end{equation*}
    where $G=(G_1,\dots,G_n)$ is a $C^\ell$-Carath\'eodory function that eliminates all the Taylor terms of order $k$ on every position of $F$ and keeps all other Taylor terms fixed up to order $|\k|-1$, i.e., for a multi index $\m\in \mathbb{N}_0^n$:
    \begin{equation*}
        D^\m_2G_i(t,0)=\left\{ \begin{array}{lcc}
             D^\m_2F_i(t,0)&  \text{ for } &   |\m|\leq k, \\
             \\ 0 & \text{ for } & |\m|=k.
             \end{array}
             \right.
    \end{equation*}
\end{corollary}

\subsection{The uniformly admissible case}

Now we study what happens when we replace condition (H2) with (H2').

\begin{lemma}\label{534}
Suppose conditions (H1) and (H2') hold.  Suppose as well that the $\k$-th condition of nonresonance on position $j$ holds. Then, there is a uniform $C^\ell$-local equivalence between (\ref{516}) and
    \begin{equation*}
        \dot{x}=A(t)x+G(t,x),
    \end{equation*}
    where $G=(G_1,\dots,G_n)$ is a $C^\ell$-Carath\'eodory function that eliminates the $\k$-th Taylor term on position $j$ of $F$ and keeps all other Taylor terms fixed up to order $|\k|$, i.e., for any other multi index $\m\in \mathbb{N}_0^n$ with $1\leq |\m|\leq |\k|$:
    \begin{equation*}
        D^\m_2G_i(t,0)=\left\{ \begin{array}{lcc}
             D^\m_2F_i(t,0)&  \text{ for } &   \m\neq \k\,\, \quad\text{or}\quad i\neq j, \\
             \\ 0 & \text{ for } & \m=\k\quad\text{ and }\quad i= j.
             \end{array}
             \right.
    \end{equation*}

    Moreover, there is a uniformly $\mu$-admissible map $\widetilde{\psi}$ such that
    $$\norm{D_2^mG(t,0)}\leq \widetilde{\psi}(t),\quad\,\forall\,m=2,\dots,\ell\text{ and almost all }t\in \mathbb{R}.$$
\end{lemma}

\begin{proof}
The argumentation follows the same steps as the proof of Lemma \ref{509}, hence we use the same notation for the functions $h$, $H$ and others. There are only two aspects that require verification:
\begin{itemize}
    \item that the local $C^\ell$-equivalence is uniform. This is an immediate consequence of (H2'), since by definition, now the map $\zeta$ defined on (\ref{510}) is bounded,
    \item that all derivatives of $G$ (including orders $|\k|+1,\dots,\ell$) are dominated by a uniformly $\mu$-admissible map on the origin. 
\end{itemize}

The second part requires a more developed argument. We carry it out in several steps.
\begin{itemize}
    \item {\it Step 1: Rectification of estimations.}

\smallskip

This is a simple observation. Let us call $M=\sup_{t\in \mathbb{R}}\zeta(t)$, where $\zeta$ is defined on (\ref{510}). From (H2'), we have $M<\infty$. Moreover, from (\ref{501}) we obtain
\begin{equation}\label{511}
    \norm{h(t,x)}\leq \frac{K^{|\k|+1}M}{\k!}\norm{x}^{|\k|}.
\end{equation}

We deduced at the end of Step 3 of the proof of Lemma \ref{509} that
$$D^m_2h(t,0)=0,\quad\forall\,m\neq |\k|,\,\text{ and all }t\in \mathbb{R},$$
and now we have
\begin{equation}\label{512}
    \norm{D^{|\k|}_2h(t,0)}=\norm{D^{|\k|}_2h(t,0)}\leq n^{|\k|}M,\quad\forall\,t\in \mathbb{R}.
\end{equation}

Furthermore, from (\ref{502}) we have
$$\norm{D_2h(t,x)}\leq nK^{|\k|+1}M\norm{x}^{|\k|-1}\,.$$

\smallskip

\item {\it Step 2: There is a tubular neighborhood of the origin which is invariant under $h$. }

\smallskip

Consider $\rho>0$. If $\norm{x}\leq \rho$, by (\ref{511}) we have 
$$\norm{h(t,x)}\leq \frac{K^{|\k|+1}M\rho^{|\k|}}{\k!},$$
thus, if we take 
$$\rho\leq \min\left\{\frac{1}{2},\left(\frac{\k!}{K^{|\k|+1}M}\right)^{\frac{1}{|\k|-1}},\frac{1}{2}\left(\frac{1}{nK^{\k|+1}M}\right)^{\frac{1}{|\k|-1}}\right\},$$
we obtain 
$$\norm{x}\leq \rho\Rightarrow \norm{h(t,x)}\leq \rho,\quad \norm{D_2h(t,x)}\leq \frac{1}{2}\quad\text{and}\quad\norm{h(t,x)}\leq \frac{1}{2}\norm{x}.$$

Now we replace the trumpet neighborhood of Step 4 of the proof of Lemma \ref{509} with this tubular neighborhood. After that we follow Steps 5-11 from the proof of Lemma \ref{509} considering this replacement.

\smallskip

    \item {\it Step 3: There is a continuous map $\vartheta:\Dom_\vartheta\subset \mathbb{R}\times \mathbb{R}^d\to \mathbb{R}^d$ such that
    $$H^{-1}(t,x)=x-h(t,x)+\vartheta(t,x),$$
    and
    \begin{equation}\label{526}
    \lim_{x\to 0}\frac{\norm{\vartheta(t,x)}}{\norm{x}^{|\k|^2-1}}=0\,,\quad \text{ uniformly on }t\in \mathbb{R}.
\end{equation}}

\smallskip

If $\norm{x}\leq \rho$, we can recursively we define the iterations of $(-h)$ by 
$$(-h)^0(t,x)=x\qquad\text{ and }\qquad(-h)^{i+1}(t,x)=-h(t,(-h)^i(t,x)),$$ 
and it is easily followed that $\norm{(-h)^i(t,x)}\leq 2^{-i}\norm{x}$, hence its geometric series converges absolutely. In other words, the series
$$\sum_{i=0}^\infty (-h)^i(t,x)=:\Tilde{h}(t,x),$$
converges absolutely to an element $\tilde{h}(t,x)$, whose norm is not greater than $2\norm{x}$. Moreover, if $\norm{x}\leq \rho$ we have:
\begin{align*}
    [\Id+h(t,\cdot)]\circ [\Tilde{h}(t,\cdot)](x)=x=[\tilde{h}(t,\cdot)]\circ[\Id+h(t,\cdot)](x),
\end{align*}
which implies $H^{-1}(t,x)=\Tilde{h}(t,x)$ for every $t\in \mathbb{R}$. Now, if we define $\vartheta:\Dom_\vartheta\subset \mathbb{R}\times \mathbb{R}^d\to \mathbb{R}^d$ by
$$\vartheta(t,x)=\sum_{i=2}^\infty (-h)^i(t,x),$$
we can write
\begin{equation}\label{527}
    H^{-1}(t,x)=x-h(t,x)+\vartheta(t,x),
\end{equation}
which clearly shows that $\vartheta$ is a continuous function. Now, from (\ref{511}) we have
\begin{align*}
    \norm{(-h)^i(t,x)}&\leq \frac{MK^{|\k|+1}}{\k!} \norm{(-h)^{i-1}(t,x)}^{|\k|}\\
   \\ &\leq \left[\frac{MK^{|\k|+1}}{\k!}\right]^{|\k|+1} \norm{(-h)^{i-2}(t,x)}^{|\k|^2}\\
    \\&\leq \left[\frac{MK^{|\k|+1}}{\k!}\right]^{|\k|+1} \left(\frac{1}{2}\right)^{(i-2)|\k|^2}\norm{x}^{|\k|^2},
\end{align*}
hence
\begin{equation*}
    \norm{\vartheta(t,x)}\leq \left[\frac{MK^{|\k|+1}}{\k!}\right]^{|\k|+1}\frac{1}{1-\left(\frac{1}{2}\right)^{|\k|^2}}\norm{x}^{|\k|^2},
\end{equation*}
which implies (\ref{526}).

\smallskip

\item {\it Step 4: For $m=2,\dots, |\k|-1$ we have $D_2^mH^{-1}(t,0)=0$, $\norm{D^{|\k|}_2H^{-1}(t,0)}\leq n^{|\k|}M$ and for $m=|\k|+1,\dots,\ell$ there is a constant $\widetilde{M}>0$ such that $\norm{D^{m}_2H^{-1}(t,0)}\leq \widetilde{M}.$}

\smallskip

The first two statements follow trivially from (\ref{527}), (\ref{526}) and (\ref{512}). Now, for $m=|\k|+1,\dots,\ell$ it is clear that $D^{m}_2H^{-1}(t,0)=D^m_2\vartheta(t,0)$.

\smallskip

Chose now $p\in \mathbb{N}$ such that $|\k|^p-1>\ell$. We have
$$\vartheta(t,x)=\sum_{i=2}^{p-1} (-h)^i(t,x)+\sum_{i=p}^\infty (-h)^i(t,x).$$

In the same fashion as we proved (\ref{526}), we have 
$$\lim_{x\to 0}\frac{1}{\norm{x}^\ell}\norm{\sum_{i=p}^\infty (-h)^i(t,x)}=0,$$
thus, denoting $\widetilde{\vartheta}(t,x)=\sum_{i=2}^{p-1} (-h)^i(t,x)$, we have  $D^{m}_2\vartheta(t,0)=D^m_2\widetilde{\vartheta}(t,0)$ for $m=|\k|+1,\dots,\ell$.  As $\widetilde{\theta}$ is a finite sum of finite compositions of $h$, which has all of its derivatives bounded in the origin, there is some $\widetilde{M}>0$ such that 
$$\norm{D^{m}_2H^{-1}(t,0)}=\norm{D^m_2\widetilde{\vartheta}(t,0)}\leq \widetilde{M},\quad\forall\,m=|\k|+1,\dots,\ell.$$

\item {\it Step 5: There is a uniformly $\mu$-admissible map $\widehat{\psi}$ such that
    $$\norm{D_2^mR(t,0)}\leq \widehat{\psi}(t),\quad\,\forall\,m=2,\dots,\ell\text{ and almost all }t\in \mathbb{R},$$
    where $R$ is the map defined on Step 10 of the proof of Lemma \ref{509}.}

\smallskip

By the previous step, the maps
$$x\mapsto D_2h_j\left(t,H^{-1}(t,x)\right)
\quad\text{and}\quad x\mapsto [x]^\k-\left[H^{-1}(t,x)\right]^{\k}\,,$$
have all their derivatives of order between $2$ and $\ell$ uniformly bounded on the origin. Then, by the definition of $R$, all of its derivatives of order between $2$ and $\ell$ are bounded by a positive linear combination of the derivatives of $F$, which are dominated by a uniformly $\mu$-admissible function by hypothesis (H2'). 
\end{itemize}

By the definition of $G$ given on Step 11 of the proof of Lemma \ref{509}, we complete this demonstration.
    
\end{proof}

We conclude this section by presenting an immediate result achieved through the repetition of the preceding lemma. The statement presented is the Theorem of Normal Forms for systems that exhibit $\mu$-dichotomies.

\begin{theorem}
Suppose conditions (H1), (H2') and (H3) hold. Then, there is a uniform $C^\ell$-local equivalence between (\ref{516}) and
    \begin{equation*}
        \dot{x}=A(t)x+G(t,x),
    \end{equation*}
    where $G=(G_1,\dots,G_n)$ is a $C^\ell$-Carath\'eodory function that eliminates all Taylor terms in the origin up to order $\ell$,  i.e.,
    $$D_2^mG(t,0)=0,\quad\,\forall \,m=2,\dots,\ell, \text{ and almost all }t\in \mathbb{R}.$$
\end{theorem}

\section{Remarks on the nonuniform case}

We finish this work by giving guidelines on how a result of elimination of nonresonant terms could work for the case where the linear part admits only a nonuniform bounded growth, as well as some of the difficulties that arise. The first of these is that, unlike the uniform case discussed on Remark \ref{533}, a conjugation by a nonuniform Lyapunov function does not preserve condition (H2) in general.

\smallskip

Assume nevertheless that we have a setting where the linear part is block diagonalized and (H2) is verified. Let $\epsilon>0$. For each spectral interval $\lambda_i=[a_i,b_i]$ we choose two numbers $\a_i=\a_i(\epsilon)$ and $\b_i=\b_i(\epsilon)$  with
$$a_i-\epsilon\leq \a_i<a_i\qquad \text{and}\qquad b_i<\b_i\leq b_i+\epsilon.$$

Similarly as before, as $\Sigma_\NmuD(A_i)=\lambda_i=[a_i,b_i]$, from the simple fact that $\b_i>b_i$ and $\a_i<a_i$ it follows that exists constants $K=K(\epsilon)\geq 1$ and 
\begin{equation}\label{528}
\alpha_i=\alpha_i(\epsilon)<0, \quad\beta_i=\beta_i(\epsilon)>0, \quad\theta_i=\theta_i(\epsilon),\quad\nu_i=\nu_i(\epsilon)\geq 0,   
\end{equation}
with 
\begin{equation}\label{529}
  \alpha_i+\theta_i<0\quad\text{ and }\quad\beta_i-\nu_i>0 , 
\end{equation}
such that for all $i=1,\dots,n$ the following estimations follow 
\begin{align*}
    \norm{\Phi_i(t,s)\left(\frac{\mu(t)}{\mu(s)}\right)^{-\b_i}}\leq K\left(\frac{\mu(t)}{\mu(s)}\right)^{\alpha_i}\mu(s)^{\sgn(s)\theta_i}\,,&\quad \text{ for all }t\geq s,\\
    \norm{\Phi_i(t,s)\left(\frac{\mu(t)}{\mu(s)}\right)^{-\a_i}}\leq K\left(\frac{\mu(t)}{\mu(s)}\right)^{\beta_i}\mu(s)^{\sgn(s)\nu_i}\,,&\quad \text{ for all }t\leq s,
\end{align*}
thus
\begin{align*}
    \norm{\Phi_i(t,s)}\leq K\left(\frac{\mu(t)}{\mu(s)}\right)^{\alpha_i+\b_i}\mu(s)^{\sgn(s)\theta_i}\,,&\quad \text{ for all }t\geq s,\\
    \norm{\Phi_i(t,s)}\leq K\left(\frac{\mu(t)}{\mu(s)}\right)^{\beta_i+\a_i}\mu(s)^{\sgn(s)\nu_i}\,,&\quad \text{ for all }t\leq s.
\end{align*}

Allow us to introduce $\eta^{+}(\epsilon)=\eta^+:\mathbb{R}\to \mathbb{R}$ by
$$\eta^+(t)=\sup_{s\geq t}\left(\frac{\mu(t)}{\mu(s)}\right)^{\beta_j-\left[k_1\alpha_1+\cdots+ k_n\alpha_n\right]}\cdot \mu(s)^{\sgn(s)\nu_j}\cdot\mu(t)^{\sgn(t)(k_1\theta_1+\cdots+k_n\theta_n)},$$
and $\eta^{-}(\epsilon)=\eta^-:\mathbb{R}\to \mathbb{R}$ by
$$\eta^-(t)=\sup_{s\leq t}\left(\frac{\mu(t)}{\mu(s)}\right)^{\alpha_j-\left[k_1\beta_1+\cdots+ k_n\beta_n\right]}\cdot \mu(s)^{\sgn(s)\theta_j}\cdot\mu(t)^{\sgn(t)(k_1\nu_1+\cdots+k_n\nu_n)}.$$

From (\ref{528}) and (\ref{529}) it follows that both $\eta^+$ and $\eta^-$ are well defined and continuous. However, unlike the uniform case discussed on the previous section, we cannot just estimate these functions by a constant. On the other hand,  other than (\ref{528}) and (\ref{529}), we do not know much about the chosen constants in general, hence we ignore if $\eta^+$ or $\eta^-$ are increasing, decreasing or have some other special properties.

\smallskip

Following the same argument as in the previous section, we can show for $t\leq s$ 
\begin{align*}
\norm{\Phi_j(t,s)\frac{1}{\k!}D_x^{\k}F_j(s,0)\cdot[\Phi(s,t)x]^{\k}}\leq &\frac{\psi(s)K^{|\k|+1}}{\k!}\norm{x_1}^{k_1}\cdots \norm{x_n}^{k_n}\\
&\cdot\left(\frac{\mu(t)}{\mu(s)}\right)^{\beta_j+\a_j-\left[k_1(\alpha_1+\b_1)+\cdots+ k_n(\alpha_n+\b_n)\right]}\\
&\cdot \mu(s)^{\sgn(s)\nu_j}\cdot\mu(t)^{\sgn(t)(k_1\theta_1+\cdots+k_n\theta_n)}\\
\leq&\frac{\psi(s)  \eta^+(t) K^{|\k|+1}}{\k!}\norm{x_1}^{k_1}\cdots \norm{x_n}^{k_n}\\
&\cdot\left(\frac{\mu(t)}{\mu(s)}\right)^{a_j-\left[k_1b_1+\cdots+ k_nb_n\right]-(|\k|+1)\epsilon},
\end{align*}
and for $t\geq s$ we have
\begin{align*}
    \norm{\Phi_j(t,s)\frac{1}{\k!}D_x^{\k}F_j(s,0)\cdot[\Phi(s,t)x]^{\k}}\leq& \frac{\psi(s)K^{|\k|+1}}{\k!}\norm{x_1}^{k_1}\cdots \norm{x_n}^{k_n}\\
    &\cdot\left(\frac{\mu(t)}{\mu(s)}\right)^{\alpha_j+\b_j-\left[k_1(\beta_1+\a_1)+\cdots+ k_n(\beta_n+\a_n)\right]}\\
&\cdot \mu(s)^{\sgn(s)\theta_j}\cdot\mu(t)^{\sgn(t)(k_1\nu_1+\cdots+k_n\nu_n)}\\
\leq&\frac{\psi(s)  \eta^-(t) K^{|\k|+1}}{\k!}\norm{x_1}^{k_1}\cdots \norm{x_n}^{k_n}\\
&\cdot\left(\frac{\mu(t)}{\mu(s)}\right)^{b_j-\left[k_1a_1+\cdots+ k_na_n\right]+(|\k|+1)\epsilon}.
\end{align*}

Thus, if we choose $\epsilon$ to be small enough, we can once again define the map $h$ and obtain
    \begin{equation*}
 \norm{h(t,x)}\leq\left\{ \begin{array}{lcc}
            
             \dfrac{K^{|\k|+1}}{\k!} \eta^+(t) \mu(t)^{\frac{1}{2}\dist(j,\k)}\zeta_{\psi,\mu,\frac{1}{2}\dist(j,\k)}^+(t)\norm{x_1}^{k_1}\cdots\norm{x_n}^{k_n}&  \text{ if  (\ref{519}) holds},\\
             \\\dfrac{K^{|\k|+1}}{\k!}\eta^-(t)\mu(t)^{-\frac{1}{2}\dist(j,\k)}\zeta_{\psi,\mu,\frac{1}{2}\dist(j,\k)}^-(t)\norm{x_1}^{k_1}\cdots\norm{x_n}^{k_n}&  \text{ if  (\ref{520}) holds}.
             \end{array}
             \right.
\end{equation*}

Now we can proceed as in the proof of Lemma \ref{509}, following Steps 2-11.  Nevertheless, by the rather uncontrollable nature of the maps $\eta^+$ and $\eta^-$, we do not know if the trumpet neighborhoods are significantly reduced. On the other hand, in this context the uniformly admissible case is not able to turn these trumpets into tubular neighborhoods, unlike the framework discussed in Lemma \ref{534}.

\end{document}